\documentclass[12pt]{amsart}
\usepackage{amsmath,amsthm,amsfonts,amssymb,amscd,euscript}

\newlength{\defbaselineskip}
\theoremstyle{plain}

\theoremstyle{definition}

\newtheorem{theoremalpha}{Theorem}

\theoremstyle{plain}
\newtheorem{thm}{Theorem}

\newtheorem{lem}[thm]{Lemma}

\newtheorem{prop}[thm]{Proposition}

\theoremstyle{definition}
\newtheorem{defn}[thm]{Definition}

\newtheorem{exmp}[thm]{Example}

\newtheorem{rem}[thm]{Remark}
\numberwithin{equation}{section}
\begin{document}





\title[The principal component of $\text{Hilb}^n(\mathbb{C}^d)$]{The singularities of
the principal component of the Hilbert scheme of points}
\author{Kyungyong Lee}


\address{Department of Mathematics, University of Michigan, Ann Arbor, MI 48109}
\address{Department of Mathematics, Purdue University, West Lafayette, IN 47907}
\email{{\tt kyungl@purdue.edu}}


 \maketitle



\section{Introduction}
Hilbert schemes of points have a rich literature in algebraic geometry, commutative algebra, combinatorics,
representation theory, and approximation theory. Various aspects of them have been studied in many contexts.
In this paper we study the local equations and the singularities of $\text{Hilb}^n(\mathbb{C}^d)$. For a
general introduction to
the field, see \cite[Chapter 18]{MS:comb}.

In \cite{H:hil}, Haiman proved the remarkable result that the isospectral Hilbert scheme
of points in the plane is normal, Cohen-Macaulay and Gorenstein.
He also showed that this implies the $n!$ conjecture and the positivity conjecture for
the Kostka-Macdonald coefficients. In addition, he conjectured that
the isospectral Hilbert scheme over the principal component of
$\text{Hilb}^n(\mathbb{C}^d)$ is Cohen-Macaulay for any $d, n\geq 1$. In particular, his
conjecture implies that the principal component of $\text{Hilb}^n(\mathbb{C}^d)$ is
Cohen-Macaulay (see \cite[Section 5.2]{H:hil} and \cite[Conjecture 18.38]{MS:comb}).

We provide a counterexample to the conjecture. The idea is to look at the local neighborhood
near $\mathfrak{m}^2$ on the principal component of $\text{Hilb}^{9}(\mathbb{C}^8)$, which is an affine cone over a
certain projective variety.
We will see that its local equations contain generators of high degree. Then the geometry of
the projective variety implies that its affine cone is not Cohen-Macaulay. Our main result is the following:

\begin{theoremalpha}\label{mainthmprinhilb}
The principal component of  $\text{Hilb}^{9}(\mathbb{C}^8)$ is not locally Cohen-Macaulay at $\mathfrak{m}^2$.
\end{theoremalpha}

Vakil showed that a number of important moduli spaces satisfy Murphy's law, and many others studied
badly-behaved moduli spaces of positive-dimensional objects (see \cite{V:bad} and the references therein). However very little is
known about how bad the singularities of the Hilbert scheme of points on a smooth variety of
dimension $>2$ can be. On the other hand, Haiman \cite[Proposition 2.6 and Remark.(2) in p.213]{MH:tq} showed that
a certain blow-up of $\text{Sym}^n(\mathbb{C}^d)$ is the principal component
of $\text{Hilb}^n(\mathbb{C}^d)$, and Ekedahl and Skjelnes \cite{ES:good} generalized it
to the case of quasi-projective schemes. If $d=2$ then the blow-up is a resolution of singularities, but
Theorem~\ref{mainthmprinhilb} implies that if $d,n \gg 0$ then the blow-up destroys the Cohen-Macaulayness
of $\text{Sym}^n(\mathbb{C}^d)$.

Turning to a more detailed description, we consider the Hilbert scheme $\text{Hilb}^{d+1}(\mathbb{C}^d)$ of
$(d+1)$ points in affine $d$-space $\mathbb{C}^d$, because it contains the squares $\mathfrak{m}^2$ of
maximal ideals. It parameterizes
the ideals $I$ of colength $(d+1)$ in
$\mathbb{C}[\mathbf{x}]=\mathbb{C}[x_1, ... ,x_d]$.

Let $V_d\subset \text{Hilb}^{d+1}(\mathbb{C}^d)$ denote the
affine open subscheme consisting of all ideals $I\in
\text{Hilb}^{d+1}(\mathbb{C}^d)$ such that $\{1,x_1, ... ,x_d\}$ is
a $\mathbb{C}$-basis of
 $\mathbb{C}[\mathbf{x}]/I$. We will call $V_d$ \emph{the symmetric affine subscheme}.
 We note that the square of any maximal
 ideal in $\mathbb{C}[\mathbf{x}]$ belongs to the symmetric affine subscheme. One may think of $V_d$ as
a deformation space of $\mathfrak{m}^2$.

 The following proposition is probably well-known to experts \cite[Section 6]{GLS:elem}, \cite{MEHu:elem}.

 \begin{prop}
Let $d\geq 2$. Let $V_d$ be the symmetric affine open subscheme of
$\emph{Hilb}^{d+1}(\mathbb{C}^d)$. Then $V_d$ is isomorphic to
$$\mathbb{C}^d \times \emph{Spec} (R_d/I_d),$$
where $R_d$ is a $d({{d+1}\choose 2}-1)$-dimensional polynomial ring and
$I_d$ is a homogeneous ideal generated by certain quadratic
polynomials. $($When $d=2$, $I_2$ is the zero ideal $(0)$.$)$
\end{prop}

More precisely, since $V_d$ admits a natural action of $GL(d)$, we can describe the quotient ring in terms of Schur
functors.

\begin{thm}\label{mainthm1}
Let $d\geq 3$. Then $V_d$ is isomorphic to
$$\mathbb{C}^d \times\emph{Spec}
\frac{\emph{Sym}^{\bullet}(\mathbb{S}_{(3,1,1,\cdots,1,0)}W)}{<\mathbb{S}_{(4,3,2,\cdots,2,1)}W>},$$where
$W$ is a $d$-dimensional $\mathbb{C}$-vector space,
$(3,1,1,\cdots,1,0)$ is a partition of $(d+1)$ and
$(4,3,2,\cdots,2,1)$ is of $(2d+2)$.
\end{thm}

If $d\leq 6$ then $V_d$ is irreducible \cite{EI:zero}, \cite{S:def}, \cite{CEVV}.
However if $d\geq 7$ then $V_d$ is reducible, and there is a
distinguished component called a principal component. For any $d$, let $P_d$ denote the principal component
of $V_d$. Here we regard
it as its reduced structure.\footnote{It is not known whether $\text{Hilb}^n(\mathbb{C}^d)$ is reduced or
not, for $d\geq 3$.} The general elements in $P_d$ are
radical ideals defining $(d+1)$ distinct points whose linear span is
non-degenerate, i.e. there is no hyperplane passing through them in
$\mathbb{C}^d$. The most special elements in $P_d$ are $\mathfrak{m}^2$. It is clear that the dimension of the principal
component is $d(d+1)$.

Let $J_d$ denote the defining ideal of $P_d$, in other words,
$$
P_d \text{ }\cong  \text{ }\mathbb{C}^d \times \text{Spec} (R_d/J_d),
$$
where $J_d$ is a reduced homogeneous ideal. 

There has been some interest in trying to find the equations $P_d$ satisfy
 (e.g. \cite[Problem 18.40]{MS:comb}, \cite[Remark 3.4]{S:limit}). But up to now
 they have not been known to satisfy any other equations, besides the quadratic Pl\"{u}cker relations.
  We present some new equations and obtain the following result.

 \begin{thm}\label{eightp}
Let $d=8$ and let $P_8$ be the principal component of $V_8$. Then
$P_8$ is isomorphic to
$$\mathbb{C}^8 \times \emph{Spec} (R_8/J_8),$$
where $R_8$ is a $8({{9}\choose 2}-1)$-dimensional polynomial ring and a
set of the minimal homogeneous generators of $J_8$ contains certain
polynomials of degree $90$. In particular, the Castelnuovo-Mumford
regularity of $J_8$ is $\geq 90$, while the dimension of $\emph{Proj}
(R_8/J_8)$ is $63$.
\end{thm}

Again more precisely,
\begin{prop}\label{mainthm2}
The principal component $P_8$ is isomorphic to
$$\mathbb{C}^8 \times\emph{Spec}
\frac{\emph{Sym}^{\bullet}(\mathbb{S}_{(3,1,1,\cdots,1,0)}W)}{J_8},$$where
the vector space of the minimal homogeneous generators of $J_8$ contains
$$\mathbb{S}_{(133,130,126,122,119,60,60,60)}W.$$
\end{prop}

\begin{proof}[Proof of Theorem~\ref{mainthmprinhilb}]
Together with the following lemma and proposition, Theorem~\ref{eightp} implies that the principal
component $P_8$ is not Cohen-Macaulay. More concretely, if $P_8$ were Cohen-Macaulay,
then $\text{Proj} (R_8/J_8)$ would be arithmetically
Cohen-Macaulay, but then Lemma~\ref{rationalacm} and Proposition~\ref{rationx} would
imply $\text{reg}(\text{Proj} (R_8/J_8))\leq 64$,
which would contradict Theorem~\ref{eightp}.
\end{proof}

\begin{lem}\label{rationalacm}
Let $S \subset \mathbb{P}^N$ be a projective arithmetically
Cohen-Macaulay variety of dimension $n$. Suppose that there is a
smooth open set $\tilde{U} \subset S$ such that\\
\noindent $\bullet$ $\emph{codim}_S(S\setminus \tilde{U}) \geq 2$, and\\
\noindent $\bullet$ $\tilde{U}$ is
covered by rational proper curves, i.e., for any point $x\in
\tilde{U}$, there is a smooth irreducible rational proper curve
on $\tilde{U}$ passing through $x$.

Then $\emph{reg}(S)\leq n+1$.
\end{lem}

\begin{prop}\label{rationx}
Let $X=\emph{Proj} (R_d/J_d)$ for $2\leq d\leq 8$. Then there is a
smooth open set $\tilde{U}_d \subset X$ such that\\
\noindent $\bullet$ $\emph{codim}_X (X\setminus \tilde{U}_d) = 2$, and\\
\noindent $\bullet$ $\tilde{U}_d$ is
covered by rational proper curves.
\end{prop}

I am grateful to Professors Rob Lazarsfeld, David Eisenbud, William Fulton, Mark Haiman,
Anthony Iarrobino, Steve Kleiman, Ezra Miller, Mircea Musta\c{t}\u{a}, Bjorn Poonen, Mihnea Popa,
Boris Shekhtman, Roy Skjelnes,
Bernd Sturmfels, Ravi Vakil, and Dustin Cartwright for their valuable advices, suggestions, comments,
discussions and correspondence.



\section{Local equations of the Hilbert scheme of points}\label{proofmain1}

In this section we prove Theorem ~\ref{mainthm1}. In fact the defining ideal of $V_d$ will be obtained
by very concrete computations.

Before we begin the proof, let us explain the notation more precisely. By Lemma \ref{comput},
there is an injective homomorphism
 $$j:\mathbb{S}_{(4,3,2,\cdots,2,1)}W \hookrightarrow \text{Sym}^{2}(\mathbb{S}_{(3,1,1,\cdots,1,0)}W)$$
of Schur modules. Then $j$ induces natural maps
$$\aligned &\mathbb{S}_{(4,3,2,\cdots,2,1)}W \otimes \text{Sym}^{r-2}(\mathbb{S}_{(3,1,1,\cdots,1,0)}W)\\
&\hookrightarrow
\text{Sym}^{2}(\mathbb{S}_{(3,1,1,\cdots,1,0)}W)\otimes
\text{Sym}^{r-2}(\mathbb{S}_{(3,1,1,\cdots,1,0)}W)\\
&\rightarrow \text{Sym}^{r}(\mathbb{S}_{(3,1,1,\cdots,1,0)}W),\text{
}\text{ }\text{ }\text{ }\text{ }\text{ }\text{ }\text{ }\text{
}\text{ }\text{ }\text{ }\text{ }\text{ }\text{ }\text{ }\text{
}\text{ }\text{ }\text{ }\text{ }\text{ }\text{ }\text{ }\text{
}\text{ }\text{ }\text{ }\text{ }\text{ }\text{ }\text{ }\text{
}\text{ }r \geq 2,\endaligned$$which define the quotient ring
$\frac{\text{Sym}^{\bullet}(\mathbb{S}_{(3,1,1,\cdots,1,0)}W)}{<\mathbb{S}_{(4,3,2,\cdots,2,1)}W>}$.

To ease notations and references, we introduce the notion of ideal
projectors(cf. \cite{B:algebra}, \cite{dB:ideal}, \cite{dB:limits},
\cite{S:limit}).

\begin{defn}
(cf. \cite{B:algebra}) A linear idempotent map $P$ on
$\mathbb{C}[\mathbf{x}]$ is called an \textbf{ideal projector} if
$\text{ker} P$ is an ideal in $\mathbb{C}[\mathbf{x}]$.
\end{defn}

We will use \emph{de Boor's formula}:

\begin{thm}\label{dB} \emph{(\cite{dB:ideal}, de Boor)} A linear mapping $P : \mathbb{C}[\mathbf{x}] \rightarrow \mathbb{C}[\mathbf{x}] $ is an
ideal projector if and only if the equality
\begin{equation}\label{deBoor}
P(gh) = P(gP(h))
\end{equation}
 holds for all $g, h \in \mathbb{C}[\mathbf{x}]$.
\end{thm}

Let $\mathcal{P}$ be the space of ideal projectors onto span
$\{1,x_1,...,x_d\}$, in other words,
$$ \mathcal{P}:=\{P : \text{ideal projector } | \text{ ker} P \in
V_d\}.
$$ The space $\mathcal{P}$ is isomorphic to the symmetric affine
subscheme $V_d$ \cite[p3]{S:bivideal}. For the sake of simplicity,
we prefer to work on $\mathcal{P}$ in place of $V_d$.

First we consider the natural embedding of $\mathcal{P}$. Gustavsen,
Laksov and Skjelnes \cite{GLS:elem} gave more general description of
open affine coverings of Hilbert schemes of points.

\begin{lem}
The space $\mathcal{P}$ can be embedded into
$\mathbb{C}^{(d+1){{d+1}\choose 2}}$.
\end{lem}
\begin{proof}[Sketch of proof]
For each ideal projector $P\in \mathcal{P}$ and each pair $(i,j)$,
$1\leq i,j\leq d$, there is a collection $p_{0,ij},p_{1,ij}, \cdots
,p_{d,ij}$ of complex numbers such that
\begin{equation}\label{pequations}
P(x_i x_j) = p_{0,ij}+\sum_{m=1}^d p_{m,ij}x_m.
\end{equation}
As $(i,j)$ varies over $1\leq i,j\leq d$, each ideal projector $P\in
\mathcal{P}$ gives rise to a collection $p_{0,ij}, p_{r, st}$
$(1\leq i,j,r,s,t \leq d)$ of complex numbers. Of course
$p_{0,ij}=p_{0,ji}$ and $p_{r,st}=p_{r, ts}$. So we have a map $f:
\mathcal{P} \rightarrow\mathbb{C}^{(d+1){{d+1}\choose
2}}=\frac{\mathbb{C}[p_{0,ij},\text{ }\text{ } p_{r, st}]_{1\leq
i,j,r,s,t \leq d}}{(p_{0,ij}-p_{0,ji},\text{ }\text{ }p_{r,st}-p_{r,
ts})}$.

Here we only show that $f$ is one-to-one. It is proved in
\cite{GLS:elem} that $f$ is in fact a scheme-theoretic embedding.

We will show that if $P_1, P_2\in\mathcal{P}$ and if
$f(P_1)=f(P_2)$, i.e. $P_1(x_i x_j)=P_2(x_i x_j)$ for every $(i,j)$,
$1\leq i,j\leq d$, then $P_1=P_2$. Since $P_1$ and $P_2$ are linear
maps, it is enough to check that
$P_1(x_{i_1}...x_{i_r})=P_2(x_{i_1}...x_{i_r})$ for any monomial
$x_{i_1}...x_{i_r}$. This follows from de Boor's formula
(\ref{deBoor}):
$$\aligned P_1(x_{i_1}...x_{i_r})&=P_1(x_{i_1}P_1(x_{i_2}\cdots
P_1(x_{i_{r-1}}x_{i_r})\cdots))\\
&=P_2(x_{i_1}P_2(x_{i_2}\cdots
P_2(x_{i_{r-1}}x_{i_r})\cdots))=P_2(x_{i_1}...x_{i_r}),
\endaligned$$where we have used the property that $P(g)$ is a linear
combination of $1,x_1,\dots,x_d$ for any $g \in
\mathbb{C}[\mathbf{x}]$.
\end{proof}

Next we describe the ideal defining $\mathcal{P}$ in
$$\frac{\mathbb{C}[p_{0,ij},\text{ }\text{ } p_{r, st}]_{1\leq
i,j,r,s,t \leq d}}{(p_{0,ij}-p_{0,ji},\text{ }\text{ }p_{r,st}-p_{r,
ts})}=:R,$$where we keep the notations in the above proof.  Let
$I_{\mathcal{P}}$ denote the ideal.

\begin{lem}\label{idealP}
Let $C(a;j,(i,k))\in R$ denote the coefficient of $x_a$ in
$$P(x_k P(x_i x_j)) - P(x_i P(x_k x_j)) \in R[x_1, \cdots,
x_d].$$ Then $I_{\mathcal{P}}$ is generated by $C(a;j,(i,k))$'s
$(0\leq a \leq d$,$\text{ }$ $1 \leq i,j,k\leq d)$. (We regard an
element in $R[x_1, \cdots, x_d]_0 \cong R$ as a coefficient of
$x_0$.)
\end{lem}

 For example, if $a\neq j,i,k$ then
\begin{equation}\label{aijkaneqijk}
C(a;j,(i,k))=\sum_{m=1}^d (p_{m,ij}p_{a,km}-p_{m,kj}p_{a,im}).\end{equation}
If $a=k$ then
\begin{equation}C(k;j,(i,k))=p_{0,ij}+\sum_{m=1}^d (p_{m,ij}p_{k,km}-p_{m,kj}p_{k,im}).\end{equation}

\begin{proof}[Proof of Lemma ~\ref{idealP}]
The de Boor's formula (\ref{deBoor}) implies that $I_{\mathcal{P}}$
is generated by coefficients of $x_a$'s ($1\leq a \leq d$) in
$P(gP(h))-P(hP(g))$ (all $g,h\in \mathbb{C}[\mathbf{x}]$). But any
$P(gP(h))-P(hP(g))$ can be generated by $P(x_k P(x_i x_j)) - P(x_i
P(x_k x_j))$'s.
\end{proof}

We note that $C(a;j,(i,k))+C(a;j,(k,i))=0$ so from now on we
identify $C(a;j,(i,k))$ with $-C(a;j,(k,i))$.

\begin{lem}
In fact, $I_{\mathcal{P}}$ is generated by $C(a;j,(i,k))$'s $(1\leq
a \leq d$,$\text{ }$ $1 \leq i,j,k\leq d)$.
\end{lem}
\begin{proof}
It is enough to prove that for any $1 \leq i,j,k\leq d$, the
polynomial $C(0;j,(i,k))$ is generated by $C(a;b,(e,f))$'s $(1 \leq
a,b,e,f\leq d)$. Fix any $u$, $1 \leq u\leq d$. Then we have
$$\aligned C(0;j,&(i,k))
=\sum_{m=1}^d
(p_{m,ij}p_{0,km}-p_{m,kj}p_{0,im})\\
=-\sum_{m=1}^d \Big(&p_{m,ij} \sum_{t=1}^d
(p_{t,km}p_{u,tu}-p_{t,ku}p_{u,tm})
 -p_{m,kj}\sum_{t=1}^d(p_{t,im}p_{u,tu}-p_{t,iu}p_{u,tm})\Big)\\
+\sum_{m=1}^d \Big(&p_{m,ij} C(u;k,(m,u))
 -p_{m,kj}C(u;i,(m,u))\Big)
  \endaligned$$
 $$\aligned
 =-\sum_{t=1}^d \Big(&p_{u,tu}\sum_{m=1}^d (p_{m,ij}p_{t,km}-p_{m,kj}p_{t,im})\\
 -&p_{t,ku}\sum_{m=1}^d(p_{m,ij}p_{u,tm}-p_{m,it}p_{u,jm})
 +p_{t,iu}\sum_{m=1}^d(p_{m,kj}p_{u,tm}-p_{m,kt}p_{u,jm})\Big)\\
 +\sum_{m=1}^d
 &p_{u,jm}\sum_{t=1}^d(p_{t,ku}p_{m,it}-p_{t,iu}p_{m,kt})\\
 +\sum_{m=1}^d \Big(&p_{m,ij} C(u;k,(m,u))
 -p_{m,kj}C(u;i,(m,u))\Big)
 \endaligned$$
 $$\aligned
 =-\sum_{t=1}^d \Big(&p_{u,tu}C(t;j,(i,k))
 -p_{t,ku}C(u;i,(j,t))
 +p_{t,iu}C(u;k,(j,t))\Big)\text{ }\text{ }\text{ }\text{ }\text{ }\text{ }\text{ }\text{ }\text{ }\text{ }\text{ }\text{ }\text{ }\\
 +\sum_{m=1}^d
 &p_{u,jm}C(m;u,(k,i))\\
 +\sum_{m=1}^d \Big(&p_{m,ij} C(u;k,(m,u))
 -p_{m,kj}C(u;i,(m,u))\Big).
 \endaligned$$
\end{proof}

So the set of generators of $I_{\mathcal{P}}$ is
$$\{C(a;j,(i,k))\text{ } | \text{ }1 \leq a,i,j,k\leq d)\}.$$
We associate to this a representation of $GL(W)$, where $W$ is a $d$-dimensional vector space.
\begin{prop}\label{vector_space}
The $\mathbb{C}$-vector space $Y$ of generators
$$\frac{<C(a;j,(i,k))\text{ } | \text{ }1 \leq a,i,j,k\leq d)>}{C(a;j,(i,k))+C(a;j,(k,i))}$$
is canonically isomorphic to
$$
\mathbb{S}_{(3,2,1,\cdots,1,0)}W \bigoplus
\mathbb{S}_{(3,1,1,\cdots,1,1)}W$$ as $\mathbb{C}$-vector spaces,
where $W$ is a $d$-dimensional vector space and $
\mathbb{S}_{(3,2,1,\cdots,1,0)}$ $($resp.
$\mathbb{S}_{(3,1,1,\cdots,1,1)})$ is the Schur functor
corresponding to the partition $(3,2,1,\cdots,1,0)$ $($resp.
$(3,1,1,\cdots,1,1))$ of $(d+2)$.
\end{prop}
\begin{proof}
Let $W=\bigoplus_{i=1}^d \mathbb{C}v_i$. Define
$$\varphi:Y \longrightarrow \bigwedge^{d-1}W \otimes W
\otimes \bigwedge^2 W$$ by
$$\varphi:C(a;j,(i,k)) \mapsto (-1)^a(v_1\wedge \cdots \wedge
\hat{v_a}\wedge \cdots \wedge v_d)\otimes v_j\otimes(v_i \wedge
v_k). $$ Then it is clear that $\varphi$ is injective.

By Littlewood-Richardson rule, we have
$$\aligned&\bigwedge^{d-1}W \otimes W \otimes \bigwedge^2 W\\
 &\cong \mathbb{S}_{(1,1,1,\cdots,1,0)}W \otimes W \otimes \mathbb{S}_{(1,1,0,\cdots,0,0)}
 W\\
 &\cong
\mathbb{S}_{(3,2,1,\cdots,1,0)}W \bigoplus
\mathbb{S}_{(3,1,1,\cdots,1,1)}W
 \bigoplus (\mathbb{S}_{(2,2,1,\cdots,1,1)}W)^{\oplus
2}\bigoplus \mathbb{S}_{(2,2,2,1,\cdots,1,0)}W\\
&\cong \mathbb{S}_{(3,2,1,\cdots,1,0)}W \bigoplus
\mathbb{S}_{(3,1,1,\cdots,1,1)}W
 \bigoplus\bigwedge^{d}W \otimes \bigwedge^2 W \bigoplus\bigwedge^{d-1}W \otimes
\bigwedge^3 W,\endaligned$$where each partition is of $(d+2)$. We
will show that the image of any element of $Y$ under $\varphi$ lies neither on
$\bigwedge^{d}W \otimes \bigwedge^2 W$ nor $\bigwedge^{d-1}W \otimes
\bigwedge^3 W$.

Since
$$
\sum_{j=1}^d (-1)^j(v_1\wedge \cdots \wedge \hat{v_j}\wedge \cdots
\wedge v_d)\otimes v_j\otimes(v_i \wedge v_k),\text{ }\text{ }\text{
}1\leq i <k\leq d,
$$
generate $\bigwedge^{d}W \otimes \bigwedge^2 W$, we need to show
that
\begin{equation}\label{eq1}
\sum_{j=1}^d C(j;j,(i,k)) =0.
\end{equation} But this is elementary
because
$$\sum_{j=1}^d C(j;j,(i,k))=\sum_{j=1}^d\sum_{m=1}^d (p_{m,ij}p_{j,km}-p_{m,kj}p_{j,im})=0.$$

Since
$$\aligned
&(v_1\wedge \cdots \wedge \hat{v_a}\wedge \cdots
\wedge v_d)\otimes v_j\otimes(v_i \wedge v_k)\\
&+(v_1\wedge \cdots \wedge \hat{v_a}\wedge \cdots
\wedge v_d)\otimes v_k\otimes(v_j \wedge v_i)\\
&+(v_1\wedge \cdots \wedge \hat{v_a}\wedge \cdots \wedge v_d)\otimes
v_i\otimes(v_k \wedge v_j),\text{ }\text{ }\text{ }1\leq a\leq
d,\text{ } 1\leq j<i<k\leq d,
\endaligned$$
generate $\bigwedge^{d-1}W \otimes \bigwedge^3 W$, we need to show
that
\begin{equation}\label{eq2}
C(a;j,(i,k))+C(a;k,(j,i))+C(a;i,(k,j)) =0.
\end{equation}
But this is again elementary because
$$
\aligned
&\sum_{m=1}^d (p_{m,ij}p_{a,km}-p_{m,kj}p_{a,im})\\
&+\sum_{m=1}^d (p_{m,jk}p_{a,im}-p_{m,ik}p_{a,jm})\\
&+\sum_{m=1}^d (p_{m,ki}p_{a,jm}-p_{m,ji}p_{a,km})=0.
\endaligned
$$

Therefore $\varphi(Y) \subset \mathbb{S}_{(3,2,1,\cdots,1,0)}W
\bigoplus \mathbb{S}_{(3,1,1,\cdots,1,1)}W$, in other words,
$$\varphi:Y \longrightarrow \mathbb{S}_{(3,2,1,\cdots,1,0)}W
\bigoplus \mathbb{S}_{(3,1,1,\cdots,1,1)}W$$ is injective.

The next lemma completes the proof.
\end{proof}

\begin{lem}
$\varphi:Y \longrightarrow \mathbb{S}_{(3,2,1,\cdots,1,0)}W
\bigoplus \mathbb{S}_{(3,1,1,\cdots,1,1)}W$ is surjective.
\end{lem}
\begin{proof}
It is enough to show that there are no nontrivial
$\mathbb{C}$-linear relations among $C(a;j,(i,k))$'s other than
$\mathbb{C}$-linear combinations of (\ref{eq1}) and (\ref{eq2}).

Suppose
\begin{equation}\label{eq3}
C(a;j,(i,k))+\sum_{u,b,e,f} c_{u,b,(e,f)}C(u;b,(e,f)) =0,\text{
}\text{ }\text{ }\text{ }\text{ }\text{ }\text{ }\text{
}c_{u,b,(e,f)}\in \mathbb{C}.
\end{equation}

If $a\neq i,j,k$ then $C(a;j,(i,k))$ contains a term
$p_{m,ij}p_{a,km}$ and a term $p_{m,kj}p_{a,im}$.  The term
$p_{m,ij}p_{a,km}$ appears only in $C(a;j,(i,k))$ and $C(a;i,(k,j))$
among all $C(u;b,(e,f))$, $1\leq u,b,e,f\leq d$. Similarly the term
$p_{m,kj}p_{a,im}$ appears only in $C(a;j,(i,k))$ and
$C(a;k,(j,i))$. So the left hand side of (\ref{eq3}) must be a
nontrivial linear combination of (\ref{eq2}) and other relations.

Similarly even if $a=i,j,$ or $k$, each term in $C(a;j,(i,k))$
appears only in the ones involved in (\ref{eq1}) or (\ref{eq2}). To
get cancelation among these, the left hand side of (\ref{eq3}) must
contain (\ref{eq1}) or (\ref{eq2}). Repeating the argument,
(\ref{eq3}) becomes a linear combination of (\ref{eq1}) and
(\ref{eq2}).
\end{proof}

The following decomposition of Schur functors will be used later.

\begin{lem}\label{comput}
We have $$\bigwedge^{d-1}W \otimes \emph{Sym}^2 W\cong
\mathbb{S}_{(2,1,1,\cdots,1,1)}W\oplus
\mathbb{S}_{(3,1,1,\cdots,1,0)}W,$$and
$$\aligned
\emph{Sym}^{2}(\mathbb{S}_{(3,1,1,\cdots,1,0)}W)\cong
&\mathbb{S}_{(6,2,2,\cdots,2,0)}W \oplus
\mathbb{S}_{(5,3,2,\cdots,1,1)}W \oplus
\mathbb{S}_{(5,2,2,\cdots,2,1)}W\\
& \oplus \mathbb{S}_{(4,4,2,\cdots,2,0)}W \oplus
\mathbb{S}_{(4,3,2,\cdots,2,1)}W \oplus
\mathbb{S}_{(4,2,2,\cdots,2,2)}W. \endaligned$$ (If $d=3$ then
$\mathbb{S}_{(5,3,2,\cdots,1,1)}W$ does not appear.)
\end{lem}
\begin{proof}
The first isomorphism follows from the Littlewood-Richardson rule.
The second isomorphism can be calculated by
\cite[pp.124--128]{R:comb}.
\end{proof}

\begin{lem}\label{firstanaly}
There is an injective homomorphism
 $$j:\mathbb{S}_{(4,3,2,\cdots,2,1)}W \hookrightarrow \emph{Sym}^{2}\Big(\bigwedge^{d-1}W \otimes \emph{Sym}^2
W\Big)$$ such that $\mathcal{P}$ $($hence the symmetric affine open subscheme
$V_d )$ is isomorphic to
$$\emph{Spec}
\frac{\emph{Sym}^{\bullet}(\bigwedge^{d-1}W \otimes \emph{Sym}^2
W)}{<\mathbb{S}_{(4,3,2,\cdots,2,1)}W>},$$ where
$(4,3,2,\cdots,2,1)$ is a partition of $(2d+2)$.
\end{lem}
\begin{proof}
Consider a diagram
$$\begin{CD} \frac{\mathbb{C}[p'_{0,ij}, \text{ }p'_{r, st}]_{1\leq i,j,r,s,t
\leq d}}{(p'_{0,ij}-p'_{0,ji}, \text{ }p'_{r,st}-p'_{r,ts})} @<f<<
\frac{\mathbb{C}[p_{0,ij}, \text{ }p_{r, st}]_{1\leq
i,j,r,s,t \leq d}}{(p_{0,ij}-p_{0,ji}, \text{ }p_{r,st}-p_{r,ts})}=:R\\
@VVgV     @.\\ T:=\frac{\mathbb{C}[p'_{r, st}]_{1\leq r,s,t \leq
d}}{(p'_{r,st}-p'_{r,ts})}  @.
\end{CD}$$
where $g$ is the natural projection and  $f^{-1}$ is defined by
$$\aligned
&p'_{0,ij} \mapsto C(i+1; j,(i,i+1)),\text{ }\text{ }\text{ }\text{
}\text{ }\text{ }\text{ }\text{ }\text{ }1\leq i\leq
j\leq d,\\
&\text{ }\text{ }\text{ }\text{ }\text{ }\text{ }\text{ }\text{
}\text{ }\text{ }\text{
}\text{ }\text{ }\text{(if } i=d \text{ then }i+1:=1) \\
 &p'_{r,st} \mapsto p_{r,st},\text{ }\text{ }\text{ }\text{ }\text{ }\text{ }\text{ }\text{ }\text{ }\text{ }\text{ }\text{ }\text{ }\text{ }\text{ }\text{ }\text{ }\text{ }\text{ }\text{ }\text{ }\text{ }\text{ }\text{ }\text{ }\text{ }\text{ }\text{ }\text{ }\text{
}1\leq r, s, t\leq d.\endaligned$$ In fact $f$
is an isomorphism because $p_{0,ij}$ is a linear term in
$$C(i+1;j,(i,i+1))=p_{0,ij}+\sum_{m=1}^d
(p_{m,ij}p_{(i+1),(i+1)m}-p_{m,(i+1)j}p_{(i+1),im}).$$

Since $C(i+1;j,(i,i+1))\in I_{\mathcal{P}}$, we have an induced
isomorphism
\begin{equation}\label{equ1}\frac{R}{I_{\mathcal{P}}}\cong\frac{T}{I_{\mathcal{P}}T},\end{equation}
where $I_{\mathcal{P}}T$ is the expansion of $I_{\mathcal{P}}$ to
$T$. We note that in this construction $C(i+1; j,(i,i+1))$ can be
replaced by any $C(k;j,(i,k))$ or $C(k;i,(j,k))$ ($k\neq i,j$),
because the resulting $I_{\mathcal{P}}T$ does not depend on the
choice $C(k;j,(i,k))$ or $C(k;i,(j,k))$. In fact this construction
is natural in the sense that we eliminate all the linear terms
appearing in $C(a;j,(i,k))$ so that the ideal $I_{\mathcal{P}}T$ is
generated by quadratic equations.

Since $p'_{0,ij}$ are eliminated under passing $g$, the direct
summand $\mathbb{S}_{(3,1,1,\cdots,1,1)}W(\cong \text{Sym}^2 W)$ in
$W$ is eliminated. Then, by Proposition~\ref{vector_space}, the
vector space of generators of $I_{\mathcal{P}}T$ is canonically
isomorphic to $\mathbb{S}_{(3,2,1,\cdots,1,0)}W$ hence to
$$\bigwedge^{d}W \otimes\mathbb{S}_{(3,2,1,\cdots,1,0)}W \cong
\mathbb{S}_{(4,3,2,\cdots,2,1)}W \subset
\text{Sym}^2\Big(\bigwedge^{d-1}W \otimes \text{Sym}^2 W\Big),$$where the
last containment follows from Lemma~\ref{comput}.

The isomorphism of rings $$T =\frac{\mathbb{C}[p'_{r, st}]_{1\leq
r,s,t \leq d}}{(p'_{r,st}-p'_{r,ts})}\cong
\text{Sym}^{\bullet}\Big(\bigwedge^{d-1}W \otimes \text{Sym}^2 W\Big)$$
naturally induces the isomorphism of quotient rings
\begin{equation}\label{equ2}\frac{T}{I_{\mathcal{P}}T} \cong \frac{\text{Sym}^{\bullet}(\bigwedge^{d-1}W \otimes \text{Sym}^2
W)}{<\mathbb{S}_{(4,3,2,\cdots,2,1)}W>}.\end{equation}Combining this
with\begin{picture}(5,5)\put(4,0){(}\end{picture}~\ref{equ1}) gives
the desired result.
\end{proof}

\begin{thm}\label{mainthm}
$\mathcal{P}$ $($hence the symmetric affine open subscheme
$V_d )$ is isomorphic to
$$\mathbb{C}^d \times\emph{Spec}
\frac{\emph{Sym}^{\bullet}(\mathbb{S}_{(3,1,1,\cdots,1,0)}W)}{<\mathbb{S}_{(4,3,2,\cdots,2,1)}W>},$$where
$(3,1,1,\cdots,1,0)$ is a partition of $(d+1)$ and
$(4,3,2,\cdots,2,1)$ is of $(2d+2)$.
\end{thm}
\begin{proof}[Sketch of Proof]
Define an isomorphism of rings $$T=\frac{\mathbb{C}[p'_{r,
st}]_{1\leq r,s,t \leq d}}{(p'_{r,st}-p'_{r,ts})}
\overset{\cong}\longrightarrow \frac{\mathbb{C}[q_{r, st}]_{1\leq
r,s,t \leq
d}}{(q_{r,st}-q_{r,ts})}=:Q$$by\\
\begin{picture}(200,30)
\put(100,12){$p'_{r,st} \mapsto $}
\put(135,12){\Big{\{}}\put(145,23){$q_{r,sr} + q_{s,ss},\text{
}\text{ }\text{ }\text{ }\text{ }\text{ }\text{ }\text{ }\text{ if }
r=t$}\put(145,2){$q_{r,st},\text{ }\text{ }\text{ }\text{ }\text{
}\text{ }\text{ }\text{ }\text{ }\text{ }\text{ }\text{ }\text{
}\text{ }\text{ }\text{ }\text{ }\text{ }\text{if } r\neq s,t.$}
\end{picture}

As a matter of fact this is a natural isomorphism, because the
square of any maximal ideal in $\mathbb{C}[\mathbf{x}]$ satisfies
$p'_{r,sr}-\frac{1}{2}p'_{s,ss}=0$ ($r\neq s$), i.e. $q_{r,sr}=0$.
It is straightforward to check that no element in minimal generators
of $I_{\mathcal{P}}Q$ contains terms involving $q_{s,ss}$, $1\leq
s\leq d$. For example, if $a,i,j,k$ are distinct, then
$$\aligned C(a;j,(i,k))&=\sum_{m=1}^d (p_{m,ij}p_{a,km}-p_{m,kj}p_{a,im})\\
&=\sum_{m\neq a,j,i,k} (p_{m,ij}p_{a,km}-p_{m,kj}p_{a,im})\\
&\text{ }\text{
}+(p_{j,ij}p_{a,kj}-p_{j,kj}p_{a,ij})+(p_{a,ij}p_{a,ka}-p_{a,kj}p_{a,ia})\\
&\text{ }\text{
}+(p_{i,ij}p_{a,ki}-p_{i,kj}p_{a,ii})+(p_{k,ij}p_{a,kk}-p_{k,kj}p_{a,ik})\endaligned$$
becomes
$$\aligned
&\sum_{m\neq a,j,i,k} (q_{m,ij}q_{a,km}-q_{m,kj}q_{a,im})\\
&\text{ }\text{
}+((q_{j,ij}+q_{i,ii})q_{a,kj}-(q_{j,kj}+q_{k,kk})q_{a,ij})+(q_{a,ij}(q_{a,ka}+q_{k,kk})-q_{a,kj}(q_{a,ia}+q_{i,ii}))\\
&\text{ }\text{
}+((q_{i,ij}+q_{j,jj})q_{a,ki}-q_{i,kj}q_{a,ii})+(q_{k,ij}q_{a,kk}-(q_{k,kj}+q_{j,jj})q_{a,ik})\\
&=\sum_{m\neq a,j,i,k} (q_{m,ij}q_{a,km}-q_{m,kj}q_{a,im})\\
&\text{ }\text{
}+(q_{j,ij}q_{a,kj}-q_{j,kj}q_{a,ij})+(q_{a,ij}q_{a,ka}-q_{a,kj}q_{a,ia})\\
&\text{ }\text{
}+(q_{i,ij}q_{a,ki}-q_{i,kj}q_{a,ii})+(q_{k,ij}q_{a,kk}-q_{k,kj}q_{a,ik}),
\endaligned$$
in which no term involves $q_{s,ss}$, $1\leq s\leq d$.

Therefore we get $$\frac{T}{I_{\mathcal{P}}T} \cong
\frac{Q}{I_{\mathcal{P}}Q} \cong \mathbb{C}[q_{s, ss}]_{1\leq s\leq
d} \otimes_{\mathbb{C}} {\frac{\mathbb{C}[q_{r, st}]_{1\leq r,s,t
\leq d,\text{ }\text{ } r\neq s\text{ or }t\neq
s}}{(q_{r,st}-q_{r,ts})}}\Big{/}{I_{\mathcal{P}}Q}.$$

On the other hand, Lemma~\ref{comput} implies
$$\text{Sym}^{\bullet}\Big(\bigwedge^{d-1}W \otimes \text{Sym}^2 W\Big)\cong
\text{Sym}^{\bullet}(\mathbb{S}_{(2,1,1,\cdots,1,1)}W\oplus
\mathbb{S}_{(3,1,1,\cdots,1,0)}W).$$ We may identify the basis of
$\mathbb{S}_{(2,1,1,\cdots,1,1)}W$ with $\{q_{s,ss} | 1\leq s\leq
d\}$. So,
by\begin{picture}(5,5)\put(4,0){(}\end{picture}~\ref{equ2}), we have
$$\aligned\frac{T}{I_{\mathcal{P}}T} &\cong
\mathbb{C}[q_{s, ss}]_{1\leq s\leq d} \otimes_{\mathbb{C}}
{\frac{\mathbb{C}[q_{r, st}]_{1\leq r,s,t \leq d,\text{ }\text{ }
r\neq s\text{ or }t\neq
s}}{(q_{r,st}-q_{r,ts})}}\Big{/}{I_{\mathcal{P}}Q}\\&\cong
\text{Sym}^{\bullet}(\mathbb{S}_{(2,1,1,\cdots,1,1)}W)\otimes
\frac{\text{Sym}^{\bullet}(\mathbb{S}_{(3,1,1,\cdots,1,0)}W)}{<\mathbb{S}_{(4,3,2,\cdots,2,1)}W>}.\endaligned$$Combining
this with\begin{picture}(5,5)\put(4,0){(}\end{picture}~\ref{equ1})
gives the desired result.
\end{proof}


\begin{exmp}\label{secondcor}
It is well known (\cite{K:desing}) that if $d=3$ then $V_d$ is
isomorphic to a cone over the Pl\"{u}cker embedding of the
Grassmannian $G(2,6)$ with a three-dimensional vertex. Let $W$ be a $3$-dimensional vector space and $W'$ a $6$-dimensional
vector space. Then
$$
\frac{\text{Sym}^{\bullet}(\mathbb{S}_{(3,1,0)}W)}{<\mathbb{S}_{(4,3,1)}W>}\text{ }\cong\text{ }
\frac{\text{Sym}^{\bullet}(\bigwedge^2 (\mathbb{S}_{(2,0,0)}W))}{<\bigwedge^4 (\mathbb{S}_{(2,0,0)}W)>}
\text{ }\cong\text{ }
\frac{\text{Sym}^{\bullet}(\bigwedge^2 W')}{<\bigwedge^4 W'>}.$$ \qed
\end{exmp}



\section{Local equations of the principal component of the Hilbert scheme of points}\label{proofmain2}

In this section, we prove Proposition~\ref{mainthm2}. We start by showing
that $J_d$ has a representation-theoretic expression.

\begin{lem}\label{directsumschur}
The $\mathbb{C}$-vector space of the minimal generators of $J_d$ is the direct sum of some irreducible Schur functors.
\end{lem}
\begin{proof}
We prove a more general statement : the vector space $(J_d)_{\leq n}:=\bigoplus_{i=0}^n (J_d)_i$ is
the direct sum of some irreducible Schur functors for every $n$.
It is enough to show that there is a group homomorphism from $GL(d)$ to $GL((J_d)_{\leq n})$ which is
comparable with the natural action of the symmetric group $S_d$.

First, there is a natural way of defining $g\cdot p_{r,st}$ for $g\in GL(d)$.
If $p_{r,st}$ is given by $I_{\{p_1,...,p_{d}\}}\in P_d$ as in (\ref{pequations}), then $g\cdot p_{r,st}$ is
given by  $I_{\{g\cdot p_1,...,g\cdot p_{d}\}}$.

Next, we define a homomorphism $\rho$ from $GL(d)$ to $GL((J_d)_{\leq n})$ as follows. For every $f(p_{r,st})_{1\leq r,s,t\leq d} \in (J_d)_{\leq n}$, define
$$
g\cdot f(p_{r,st})_{1\leq r,s,t\leq d}
$$
by
$$
g\cdot f(p_{r,st})_{1\leq r,s,t\leq d}:=f(g\cdot p_{r,st})_{1\leq r,s,t\leq d}.
$$
Since any point in $P_d$ satisfies $f=0$, we have $f(g\cdot p_{r,st})\in (J_d)_{\leq n}$. It
is easy to check that $\rho$ is a homomorphism from $GL(d)$ to $GL((J_d)_{\leq n})$. It is obvious that this
is comparable with the natural action of $S_d$.
\end{proof}

\begin{lem}
The vector space of the minimal homogeneous generators of the ideal $J_8 \subset \emph{Sym}^{\bullet}(\mathbb{S}_{(3,1,1,\cdots,1,0)}W)$ contains
$$\mathbb{S}_{(133,130,126,122,119,60,60,60)}W.$$\end{lem}
\begin{proof}[Sketch of proof]
The idea is to observe that there are relations among
$$
\{p_{r,st}\}_{1\leq r\leq 3, \text{ }4\leq s,t \leq 8}.
$$
This is suggested by the fact that a general ideal having the Hilbert function of the type $(1,5,3)$ is
not contained in the principal component $P_8$ \cite{EI:zero}, \cite{S:def}, \cite{CEVV}. In particular,
 if $
\{p_{r,st}\}_{1\leq r\leq 3, \text{ }4\leq s,t \leq 8}
$ are general complex numbers and if the other coordinates are $0$, then the colength 9 ideal determined by those
coordinates does not belong to $P_8$.

By the algorithm in \cite[pp.124--128]{R:comb}, one can check that $\mathbb{S}_{(133,130,126,122,119,60,60,60)}W$ appears in the decomposition of $\text{Sym}^{90}(\mathbb{S}_{(3,1,...,1,0)}W)$.
We find elements in $\mathbb{S}_{(133,130,126,122,119,60,60,60)}W$ in a very explicit way.

Recall from (\ref{aijkaneqijk}) that if $a\neq j,i,k$ then $$C(a;j,(i,k))=\sum_{m=1}^d (p_{m,ij}p_{a,km}-p_{m,kj}p_{a,im})$$
in the polynomial ring
$\frac{\mathbb{C}[p_{r, st}]_{1\leq r,s,t \leq d}}{(p_{r,st}-p_{r,ts})}$. The key fact is that
any term in any $C(a;j,(i,k))$ with
$1\leq a \leq 3, 4\leq i,j,k\leq d$ is a product of two coordinates, one of which is in
$\{p_{r,st}\}_{1\leq r \leq 3, 4\leq s,t,\leq d}$ and the other is not. We consider the following $90\times 1$ matrix each of whose entry is a polynomial of degree 2.
$$
\left(
                                   \begin{array}{c}
                                      C(1;4,(5,6)) \\
                                      \vdots \\
                                       C(a; j, (i,k)) \\
                                        \vdots \\
                                      C(3;6,(5,6) ) \\
                                    \end{array}
                                  \right),
$$
where $1\leq a \leq 3$, and $4\leq j<i<k\leq8$ or $4\leq k< j<i \leq 8$ or $4=i=j<k\leq 8$
or $4=i<j=k\leq 8$ or $5=i=j<k\leq 6$ or $5=i<j=k \leq 6$.

Then we can observe that there is a $90\times 115$ matrix $ \mathbf{M}$ such that each entry of $ \mathbf{M}$ is one of the elements in
$$
\{0, \pm p_{r,st} \}_{1\leq r\leq 3, \text{ }4\leq s,t \leq 8},
$$
 and  $\mathbf{M}$ fits into the following matrix factorization:
$$
\left(
                                   \begin{array}{c}
                                      C(1;4,(5,6)) \\
                                      \vdots \\
                                       C(a; j, (i,k)) \\
                                        \vdots \\
                                      C(3;6,(5,6) ) \\
                                    \end{array}
                                  \right)=
                                  \mathbf{M}\cdot
            \left(
                                   \begin{array}{c}
                                                                            \vdots \\
                                       p_{r',s't'}- \delta_{r',t'} \frac{p_{s',s's'}}{2} - \delta_{r',s'} \frac{p_{t',t't'}}{2} \\
                                        \vdots \\
                                                                          \end{array}
                                  \right),
$$
where $\delta$ denotes the Kronecker delta, and  $1\leq r',s'\leq 3<t'\leq 8$ or $4\leq r'\leq 8$, $4\leq s'\leq t'\leq 8 \text{ but }r',s',t' \text{ are not all equal}$.

Then exhaustive computations show that  the determinant of any $90\times 90$ minor
of $\mathbf{M}$  lies in $\mathbb{S}_{(133,130,126,122,119,60,60,60)}W$ and in $J_8$, and
that the determinant of some $90\times 90$ minor is nonzero.\footnote{
One can check whether or not a given polynomial belongs to $J_8$, since $P_8$ admits
an explicit rational parametrization. (For example, see \cite[Theorem 3.3]{S:limit} or
\cite[Proposition 2.6]{MH:tq}.)} Then, thanks to Lemma~\ref{directsumschur}, $J_8$
contains $\mathbb{S}_{(133,130,126,122,119,60,60,60)}W.$

It remains to show that  $\mathbb{S}_{(133,130,126,122,119,60,60,60)}W$ is contained
in the minimal generators of $J_8$.  If $\mathbb{S}_{(133,...)}W$ were not minimal, there would be
a partition $\lambda$ of $89\cdot 9$ such that
$$
\mathbb{S}_{\lambda} W \subset J_8\cap \text{Sym}^{89}(\mathbb{S}_{(3,1,...1,0)}W),
$$and $\mathbb{S}_{\lambda} W$ generates $\mathbb{S}_{(133,...)}W$. But we have to check
that there is no such $\lambda$.

To this end, we consider all the partitions $\lambda$ such that
$$
\mathbb{S}_{\lambda} W \subset (\mathbb{S}_{(3,1,...1,0)}W)^{\otimes 89}
$$
and
$$\mathbb{S}_{\lambda} W \otimes \mathbb{S}_{(3,1,...1,0)}W \supset \mathbb{S}_{(133,130,126,122,119,60,60,60)}W.
$$
There are 15 such partitions, and I checked that none of their embeddings into
$\text{Sym}^{89}(\mathbb{S}_{(3,1,...1,0)}W)$ are in $J_8$. We remark that each $89\times 89$ minor
of $\mathbf{M}$ belongs
to one of such $\mathbb{S}_{\lambda} W$.
\end{proof}

\begin{rem}
We note that the generator $\mathbb{S}_{(4,3,2,\cdots,2,1)}W$ of $I_8$ does not generate $\mathbb{S}_{(133,...)}W$, in other words,
$$
\mathbb{S}_{(133,130,126,122,119,60,60,60)}W \not\subset <\mathbb{S}_{(4,3,2,\cdots,2,1)}W>.
$$ It is an elementary consequence of the combinatorial Littlewood-Richardson rule(for example,
see \cite[p456]{FH:repre}). In fact any  $\mathbb{S}_\lambda W$ ($\lambda=(\lambda_1, \lambda_2, \cdots, \lambda_8)$) appearing in the decomposition of  $\mathbb{S}_{(4,3,2,\cdots,2,1)}W \otimes (\mathbb{S}_{(3,1,1,\cdots,1,0)}W)^{\otimes (r-2)}$ satisfies $\lambda_{8-k} +\cdots + \lambda_8\geq rk+1$, for any $r\geq 2$ and any $k=0,...,7$.

Concretely speaking, the ideal generated by $C(a; j, (i,k))$ does not contain
any nonzero determinants of $90\times 90$ minors of $\mathbf{M}$. It is easy to
prove this without using Schur functors, because for any term $\prod_i p_{r_i,s_it_i}$ in any
determinant of $90\times 90$ minors of $\mathbf{M}$, we have $r_i\neq s_j, t_j$ for all $i,j$.
\end{rem}

\section{Proof of Lemma~\ref{rationalacm}}
Lemma~\ref{rationalacm} can be considered as a standard fact. We do not claim any novelty for its proof.
\begin{proof}[Proof of Lemma~\ref{rationalacm}]
Recall that the regularity index of $S$, $r(S)$, is the minimum
degree in which the Hilbert function of $S$ agrees with the Hilbert
polynomial (see \cite{BH:CM} for more details). If $S \subset
\mathbb{P}^N$ is an aCM scheme of dimension $n$, then $r(S) =
\text{reg}(S)-n-1$ (this follows from \cite[Theorem 4.4.3
(b)]{BH:CM}). Hence it is enough to show that $r(S)\leq 0$.

The Hilbert function of $S$ is $H(S,t) = h^0(S,\mathcal{O}_S(t))$
and its Hilbert polynomial is $\chi(S,t) = h^0 (S,\mathcal{O}_S(t))
+ (-1)^n h^n(S,\mathcal{O}_S(t)) = H(S,t) + (-1)^n h^0 (S, \omega_S(-t))$
where $\omega_S$ denotes the dualizing sheaf.

The second condition on $\tilde{U}$ implies $H^0(\tilde{U}, \omega_{\tilde{U}})=0$
(see \cite[Chapter 4]{Kol:rational}). Since $S$ is Cohen-Macaulay
and $\text{codim } S\setminus \tilde{U}\geq 2$, we have $H^0(S, \omega_S)=0$. So we have $h^0(S, \omega_S(-t)) = 0$
for all $t\geq 0$ and this establishes the lemma.
\end{proof}

\section{Proof of Proposition~\ref{rationx}}\label{proofrationalx}
In this section we prove Proposition~\ref{rationx}. We first construct an open
subset $U_d$ of $\text{Spec} (R_d/J_d)$, where $\tilde{U}_d$ will be the projective
counterpart of $U_d$ in $\text{Proj} (R_d/J_d)$.

Let $U_d$ be the open subset of $P_d$ consisting of all ideals $I\in P_d$
such that the radical $\text{Rad}(I)$ of $I$ defines at least $d$
distinct points. Then $U_d$ is smooth and $\text{codim}_{P_d} (P_d\setminus
U_d)=2$.\footnote{For $2\leq d \leq 8$, it can be checked by the computer algebra system Macaulay 2.} We consider the Hilbert--Chow morphism on $U_d$,
$$
\rho : U_d \longrightarrow  \text{Sym}^{d+1}(\mathbb{C}^d),
$$
and the  averaging map
$$
\pi :  \text{Sym}^{d+1}(\mathbb{C}^d) \longrightarrow  \mathbb{C}^d
$$
given by
$$
\pi\big(\{(x_{1,1},...,x_{1,d}), ..., (x_{d+1,1},...,x_{d+1,d}) \}\big) =
\Big(\frac{x_{1,1}+...+x_{d+1,1}}{d+1}, ... ,
\frac{x_{1,d}+...+x_{d+1,d}}{d+1}\Big).
$$

Let $j$ be the natural morphism
$$
j : U_d \hookrightarrow P_d \cong \mathbb{C}^d \times \text{Spec} (R_d/J_d),
$$
and let $pr_1$ be the projection
$$
pr_1 : \mathbb{C}^d \times \text{Spec} (R_d/J_d)  \longrightarrow
\mathbb{C}^d.
$$

\begin{lem}
$pr_1|_{j(U_d)}$ agrees with $\pi \circ \rho \circ
j^{-1}|_{j(U_d)}$, in other words, the following diagram
\begin{picture}(1,1)
\put(-168,-22){\tiny{$>$}}\put(-145,-66){\tiny{$>$}}
\put(-260,-20){\line(1,0){96}}\put(-260,-64){\line(1,0){120}}
\put(-215,-16){\scriptsize{$\rho$}}\put(-202,-60){\scriptsize{$pr$}}\put(-194,-62){\tiny{$1$}}
\end{picture}
$$\begin{CD}
 \text{ }\text{ }\text{ }\text{ }\text{ }\text{ }\text{ }
 \text{ }\text{ }\text{ }\text{ }\text{ }\text{ }\text{ }
 \text{ }\text{ }\text{ }\text{ }\text{ }\text{ }\text{ }\text{ }\text{ }
 \text{ }\text{ }\text{ }\text{ }  U_d  \text{ }\text{ }\text{ }\text{ }\text{ }\text{ }\text{ }
 \text{ }\text{ }\text{ }\text{ }\text{ }\text{ }\text{ }
 \text{ }\text{ }\text{ }\text{ }\text{ }\text{ }\text{ }\text{ }\text{ }
 \text{ }\text{ }\text{ }\text{ }     @.       \emph{Sym}^{d+1}(\mathbb{C}^d)\\
      @A{j^{-1}}AA                      @VV{\pi}V\\
\mathbb{C}^d \times \emph{Spec} (R_d/J_d)\supset j(U_d)  \text{ }\text{ }\text{ }\text{ }\text{ }\text{ }\text{ }
 \text{ }\text{ }\text{ }\text{ }\text{ }\text{ }\text{ }
 \text{ }\text{ }\text{ }\text{ }\text{ }\text{ }\text{ }\text{ }\text{ }
 \text{ }\text{ }\text{ }\text{ } @. \mathbb{C}^d @.\text{ }\text{ }\text{ }\text{ }\text{ }\text{ }\text{ }\text{ }\text{ }
 \text{ }\text{ }\text{ }
\end{CD}$$
is commutative $($up to automorphisms of $\mathbb{C}^d \times \emph{Spec} (R_d/J_d))$.
\end{lem}
\begin{proof}
For each element $I$ in $U_d$, there is a corresponding ideal projector $P_I$, which gives rise to
$p_{r,sr}$, $1\leq r,s \leq d$ as in (\ref{pequations}). From the proof of Theorem~\ref{mainthm},
we may define $$pr_1 \circ j : U_d \longrightarrow \mathbb{C}^d$$ by
$$\aligned I\text{ } \mapsto\text{ } &\Big(\frac{\sum_{r=1}^d p_{r,1r}}{d+1}, \cdots, \frac{\sum_{r=1}^d p_{r,dr}}{d+1}\Big).
\endaligned$$ It is elementary to check that this map is the same as $\pi \circ \rho$.
\end{proof}

We identify the fiber $(pr_1)^{-1}(O)$ over the origin
$O=(0,...,0)\in \mathbb{C}^d$, with the affine cone over
$\text{Proj} (R_d/J_d)$. By construction, $j(U_d) \cap  (pr_1)^{-1}(O)$ parameterizes the
ideals defining $d$ distinct points and one more (possibly
infinitely near) point, whose average (=center of mass) is the
origin $O$. Hence scaling distances from $O$ by any nonzero constant
preserves membership in $j(U_d) \cap (pr_1)^{-1}(O)$. In other words,
if an ideal defining $(d+1)$ points, say $p_1,...,p_{d+1}$, belongs
to $j(U_d) \cap (pr_1)^{-1}(O)$, then so does the ideal defining
$\lambda\cdot p_1,...,\lambda\cdot p_{d+1}$ for any $\lambda\neq 0
\in\mathbb{C}$. In fact we have $$\text{Proj} (R_d/J_d)=(\text{Spec} (R_d/J_d) \setminus O
)/\sim,$$ where the equivalence relation is given by
$I_{\{p_1,...,p_{d+1}\}}\sim I_{\{\lambda\cdot p_1,...,\lambda\cdot
p_{d+1}\}}$, $\lambda\neq 0$.

Therefore all told, $j(U_d) \cap  (pr_1)^{-1}(O)$ is the affine cone
over a certain open subset of $\text{Proj} (R_d/J_d)$,
with the vertex of the cone removed. We denote the open subset by $\tilde{U}_d$. So we get
$$\tilde{U}_d\text{ }=\text{ }\left(j(U_d) \cap
(pr_1)^{-1}(O)\right)/\sim.$$ Of course, for any point $q\in \mathbb{C}^d$, we have
$$\tilde{U}_d\text{ }\cong\text{ }\left(j(U_d) \cap
(pr_1)^{-1}(q)\right)/\sim.$$

\begin{lem}
For any $d\geq 2$, $\tilde{U}_d$ is covered by rational proper curves, i.e., for any point
$x\in \tilde{U}_d$, there is a smooth irreducible rational
proper curve on $\tilde{U}_d$ passing through $x$.
\end{lem}
\begin{proof}
The idea of the proof is to find a smooth irreducible rational
proper curve on $\tilde{U}_d$, and to apply the $GL(d)$-action to the curve.

First we find a $\mathbb{P}^1$ on $\tilde{U}_d$ as follows. Consider the following $(d+1)$
points on $\mathbb{C}^d$ : $x_i=(x_{i,1},...,x_{i,d})$, $1\leq i\leq d+1$, where

\smallskip

\begin{picture}(200,30)
\put(100,12){$x_{i,j}\text{ }=$}
\put(135,12){\Big{\{}}\put(145,23){$-d,\text{
}\text{ }\text{ }\text{ }\text{ }\text{ }\text{
}\text{ }\text{ }\text{ }\text{ }\text{ }\text{ }\text{ }\text{ }\text{ if }
i=j$}\put(145,2){$1,\text{ }\text{ }\text{ }\text{ }\text{
}\text{ }\text{ }\text{ }\text{ }\text{ }\text{ }\text{ }\text{
}\text{ }\text{ }\text{ }\text{ }\text{ }\text{if } i\neq j.$}
\end{picture}

\noindent Note that the image of the ideal defining $x_1,...,x_{d+1}$ under $(pr_1 \circ j)$ is the origin $O$.

We fix $d$ distinct points $[a_j : b_j]$ $(1\leq j\leq d)$ on $\mathbb{P}^1\setminus [1:1]$, and
define a morphism
$$
\varphi : \mathbb{P}^1 \longrightarrow \tilde{U}_d
$$
by
$$
[\alpha : \beta]\text{ }\mapsto\text{ }\Big[\small{\text{the ideal vanishing
along }}\Big(\small{\frac{b_1\alpha -a_1\beta}{b_1-a_1}}x_{i,1},...,
\small{\frac{b_d\alpha -a_d\beta}{b_d-a_d}}x_{i,d}\Big), 1\leq i\leq d+1\Big],
$$ where $[$ ideal $]$ denotes the equivalence class of the ideal. For each $j$, we define $\varphi([a_j:b_j])$ by
the equivalence class of the non-radical ideal as a limit of radical ideals,
where two points collide. In other words,
$$
\varphi([a_j : b_j])\text{ }=\text{ }\text{ }\text{ }\text{ }\text{ }\text{ }\text{lim}\text{ }\text{ }\text{ }\text{ }\text{ }\text{ }\text{ }\text{ }\varphi([\alpha : \beta]).
$$\begin{picture}(1,1) \put(215,10){\tiny{$[\alpha : \beta]\rightarrow [a_j : b_j]$}}
\put(1,-4){It is straightforward to check that $\varphi$ is well-defined and
that $\varphi(\mathbb{P}^1)$ is smooth and irreducible.}\end{picture}

Now we can prove the lemma. For any element $[I]\in \tilde{U}_d$, there is $g\in GL(d)$ such that
$$
[I]\in g\cdot (\varphi(\mathbb{P}^1)),
$$
where $\cdot$ is the natural action of $GL(d)$.
\end{proof}



\begin{thebibliography}{99}


\bibitem{B:algebra}
Garrett Birkhoff, The Algebra of Multivariate Interpolation, in
\emph{Constructive approaches to mathematical models}, C.V. Coffman
and G. J. Fix (eds.), 345--363, Academic Press, New-York, 1979.

\bibitem{BH:CM}
Winfried Bruns, J\"{u}rgen Herzog,  Cohen-Macaulay rings.
Cambridge Studies in Advanced Mathematics, \textbf{39}, Cambridge University Press, Cambridge, 1993.

\bibitem{CEVV}
 Dustin A. Cartwright, Daniel Erman, Mauricio Velasco, Bianca Viray, Hilbert schemes of 8 points
 in $\mathbb{A}^d$, arXiv:0803.0341.

\bibitem{R:comb}
Young-Ming Chen, Adriano Garsia, Jeffrey Remmel, Algorithms for
plethysm. Combinatorics and algebra (Boulder, Colo., 1983),
109--153, Contemp. Math., \textbf{34}, Amer. Math. Soc., Providence,
RI, 1984.

\bibitem{dB:ideal}
Carl de Boor, Ideal Interpolation, in \emph{Approximation Theory XI,
Gatlinburg 2004}, Chui, C. K., M. Neamtu and L. Schumaker (eds.),
Nashboro Press (2005), 59--91.

\bibitem{dB:limits}
Carl de Boor, What Are the Limits of Lagrange Projectors?, in
\emph{Constructive Theory of Func- tions, Varna 2005}, B. Bojanov,
ed., Marin Drinov Academic Publishing House, Sofia, (2006), 51--63.

\bibitem{ES:good}
Torsten Ekedahl, Roy Skjelnes, Recovering the good component of the Hilbert scheme, math.AG/0405073.

\bibitem{EI:zero}
J. Emsalem, A. Iarrobino, Some zero-dimensional generic
singularities; finite algebras having small tangent space,
Compositio Math. \textbf{36} (1978), no. 2, 145--188.

\bibitem{FH:repre}
William Fulton, Joe Harris, Representation theory : A first course,
Graduate Texts in Mathematics, \textbf{129}. Readings in
Mathematics. Springer-Verlag, New York, 1991.

\bibitem{GLS:elem}
T. S. Gustavsen, D. Laksov,  R. M. Skjelnes, An elementary,
explicit, proof of the existence of Hilbert schemes of points, J.
Pure Appl. Algebra \textbf{210} (2007), no. 3, 705--720.

\bibitem{H:hil}
Mark Haiman, Hilbert schemes, polygraphs and the Macdonald
positivity conjecture, J. Amer. Math. Soc. \textbf{14} (2001), no.
4, 941--1006.

\bibitem{MH:tq}
Mark Haiman, $t, q$-Catalan numbers and the Hilbert scheme, Discrete Math. \textbf{193} (1998), 201--224.

\bibitem{MEHu:elem}
 Mark E. Huibregtse, An elementary construction of the multigraded
 Hilbert scheme of points, Pacific J. Math. \textbf{223} (2006), no. 2, 269--315.

\bibitem{I:red}
Anthony Iarrobino, Reducibility of the families of $0$-dimensional
schemes on a variety, Invent. Math. \textbf{15} (1972), 72--77.

\bibitem{K:desing}
Sheldon Katz, The desingularization of ${\rm Hilb}\sp 4P\sp 3$ and
its Betti numbers, Zero-dimensional schemes (Ravello, 1992),
231--242, de Gruyter, Berlin, 1994.

\bibitem{Kol:rational}
J\'{a}nos Koll\'{a}r, Rational curves on algebraic varieties, Springer-Verlag, Berlin, 1996.

\bibitem{MS:comb}
Ezra Miller, Bernd Sturmfels, \textit{Combinatorial commutative
algebra}, Graduate Texts in Mathematics, \textbf{227}.
Springer-Verlag, New York, 2005.

\bibitem{S:def}
I. R. Shafarevich, Deformations of commutative algebras of class
$2$. (Russian) Algebra i Analiz \textbf{2} (1990), no. 6, 178--196;
translation in Leningrad Math. J. \textbf{2} (1991), no. 6,
1335--1351


\bibitem{S:bivideal}
Boris Shekhtman, Bivariate Ideal Projectors and their Perturbations,
to appear in Advances in Computational Mathematics.

\bibitem{S:limit}
Boris Shekhtman, On the Limits of Lagrange Projectors, to appear in Journal of Approximation theory.

\bibitem{V:bad}
Ravi Vakil, Murphy's Law in algebraic geometry: Badly-behaved deformation spaces, Invent. Math.
\textbf{164} (2006), 569--590.


\end{thebibliography}
\end{document}